\documentclass[oneside,11pt]{amsart}

\usepackage{float}
\usepackage{amscd,amssymb}
\usepackage[mathscr]{eucal}
\usepackage[all]{xy}
\usepackage{mathrsfs}
\usepackage{amsfonts}
\usepackage{amsmath}
\usepackage{amsthm}
\usepackage{latexsym}
\usepackage[all]{xy}

\title{On the $\sf D$-affinity of flag varieties in positive
 characteristic}
 
\author{Alexander Samokhin}

\address{Institute for Information Transmission Problems, Moscow,
  Russia}

\email{alexander.samokhin@gmail.com}

\newcommand{\Oo}{\mathcal O}
\newcommand{\Uu}{\mathcal U}

\newcommand{\Pp}{\mathbb P}

\newcommand{\Ff}{\mathcal F}
\newcommand{\Ee}{\mathcal E}
\newcommand{\Ll}{\mathcal L}

\newcommand{\D}{\mathcal D}

\newcommand*{\Hom}{\mathop{\mathrm Hom}\nolimits}
\newcommand*{\Dd}{\mathop{\mathrm D\kern0pt}\nolimits}

\newcommand*{\Ext}{\mathop{\mathrm Ext}\nolimits}

\newtheorem{theorem}{Theorem}[section]

\newtheorem{lemma}{Lemma}[section]
\newtheorem{proposition}{Proposition}[section]

\newtheorem{remark}{Remark}[section]

\long\def\comment#1{}

\begin{document}

\maketitle

\begin{abstract}
Let ${\bf G}$ be a simple simply connected algebraic group of type
${\bf B}_2$ over an algebraically closed field $k$ of odd
characteristic. We prove that the flag variety ${\bf G}/{\bf B}$ is
$\sf D$-affine. This extends an earlier result of Andersen and Kaneda
\cite{AK}. 
\end{abstract}

\maketitle

\section{\bf Introduction}
Let $X$ be a smooth algebraic variety over an algebraically closed
field $k$, and let ${\mathcal D}_X$ be the sheaf of differential operators on
$X$. Then $X$ is said to be $\sf D$-affine if the
following two conditions hold: (i) for any ${\mathcal D}_X$-module $\rm M$ that is quasi-coherent
  over $\Oo _X$ the natural morphism 
${\mathcal D}_X\otimes _{{\Gamma ({\mathcal D}_X})}{\Gamma (\rm
  M)}\rightarrow {\rm M}$ is onto, and  (ii) ${\rm H}^i(X,{\mathcal
  D}_X) = 0$ for $i > 0$. 
Let $\bf G$ be a semisimple algebraic group over $k$ and $\bf P$ a
parabolic subgroup of $\bf G$. If $k$ is of characteristic zero
then the well--known Beilinson--Bernstein localization theorem
\cite{BB} states that ${\bf G}/{\bf P}$ is $\sf D$-affine. Much less is known when
$k$ is of characteristic $p>0$. Haastert \cite{Haas} showed that projective spaces and the flag variety of the group ${\bf SL}_3$
are $\sf D$-affine, and Langer \cite{Lan} proved the $\sf D$-affinity
for odd-dimensional quadrics if the characteristic of $k$ is greater than the dimension of
variety (while even-dimensional quadrics turn out to be not $\sf
D$-affine). Even earlier, Kashiwara and Lauritzen \cite{KaLau} produced a
counterexample to the $\sf D$-affinity: their result implies that the flag variety of the
group ${\bf SL}_5$ is not $\sf D$-affine in any characteristic. Nevertheless, the question about which
flag varieties are $\sf D$-affine in positive characteristic remains
open; nothing was known except the above cases.
In the present paper we show that the flag variety of the group ${\bf
  Sp}_4$ is $\sf D$-affine in odd characteristic. By \cite{Haas}, it is sufficient to prove
that ${\rm H}^{i}({\bf Sp}_4/{\bf B},{\mathcal D}_{{{\bf Sp}_4}/{\bf
    B}})=0$ for $i>0$. This is achieved by showing that
all the terms of the $p$-filtration on the sheaf ${\mathcal D}_{{{\bf
    Sp}_4}/{\bf B}}$ have vanishing higher cohomology groups, thus 
extending an earlier result of Andersen and Kaneda \cite{AK}, where they showed the cohomology vanishing of the
first term of the $p$-filtration (in any characteristic). Contrary to their
representation theoretic approach, we use simple geometric arguments
to reduce the problem to computing cohomology groups of line
bundles on the flag variety ${\bf Sp}_4/{\bf B}$. Cohomology of line
bundles on flag varieties in the rank two case are well understood thanks
to Andersen's et al. work (see \cite{An} for a recent survey and \cite{Jan} for a comprehensive treatment); working out the
cohomology groups in question completes the
proof. However, for the sake of consistency with our approach and convenience of the reader,
we explicitly show all the necessary vanishings without the use of
general theory.

\subsection*{Acknowledgements} This work has been done during
the author's stays at the IHES and the Institute of Mathematics of Jussieu in the
spring of 2009. He gratefully acknowledges the hospitality and
support of both institutions. Special thanks are due to Bernhard Keller
for kindly inviting the author to the Institute of Mathematics of Jussieu during last few years.

\section{\bf Preliminaries}

Let $k$ be an algebraically closed field of odd characteristic $p>0$, and
$\sf V$ be a symplectic vector space of dimension 4 over $k$. Let $\bf G$ be the
symplectic group ${\bf Sp}_4$ over $k$; the root system of $\bf G$ is of type ${\bf
  B}_2$. Let $\bf B$ be a Borel subgroup of $\bf G$. Consider the flag variety ${\bf G}/{\bf B}$. The group $\bf G$ has two parabolic subgroups ${\bf P}_{\alpha}$ and ${\bf
  P}_{\beta}$ that correspond to the simple roots $\alpha$ and
$\beta$, the root $\beta$ being the long root. 
The homogeneous spaces ${\bf G}/{{\bf P}_{\alpha}}$ and  ${\bf G}/{{\bf
    P}_{\beta}}$ are isomorphic to the 3-dimensional quadric ${\sf Q}_3$ and $\Pp
^3$, respectively. Denote $q$ and $\pi$ the two projections of ${\bf G}/{\bf B}$ onto ${\sf Q}_3$ and $\Pp
^3$. The line bundles on ${\bf G}/{\bf B}$ that correspond to the fundamental weights $\omega _{\alpha}$ and
$\omega _{\beta}$ are isomorphic to $\pi ^{\ast}\Oo _{\Pp ^3}(1)$ and
$q^{\ast}\Oo _{{\sf Q}_3}(1)$, respectively. The canonical line bundle
$\omega _{{\bf G}/{\bf B}}$ corresponds to the weight $-2\rho =
-2(\omega _{\alpha} + \omega _{\beta})$ and is isomorphic to
$\pi ^{\ast}\Oo _{\Pp ^3}(-2)\otimes q^{\ast}\Oo _{{\sf Q}_3}(-2)$.
The projection $\pi$ is the projective
bundle over $\Pp ^3$ associated to a rank two vector bundle $\sf N$
over $\Pp ^3=\Pp ({\sf V})$, and the projection $q$ is the projective bundle associated to the spinor bundle $\Uu
_2$ on ${\sf Q}_3$. The bundle $\sf N$ is symplectic, that is there is
a non-degenerate skew-symmetric pairing $\wedge ^2{\sf N}\rightarrow
\Oo _{\Pp ^3}$ that is induced by the given symplectic structure on
$\sf V$. There is a short exact sequence on $\Pp ^3$:
\begin{equation}
0\rightarrow \Oo _{\Pp ^3}(-1)\rightarrow \Omega _{\Pp ^3}^1(1)\rightarrow {\sf
  N}\rightarrow 0,
\end{equation}

while the spinor bundle $\Uu _2$, which is also isomorphic to the restriction 
of the rank two universal bundle on ${\rm Gr}_{2,4}={\sf Q}_4$ to
${\sf Q}_3$, fits into a short exact sequence on ${\sf Q}_3$:
\begin{equation}
0\rightarrow \Uu _2\rightarrow {\sf V}\otimes \Oo _{{\sf
    Q}_3}\rightarrow \Uu _2^{\ast}\rightarrow 0.
\end{equation}

Let ${\mathcal D}_{{\bf G}/{\bf B}}$ be the sheaf of differential
operators on ${\bf G}/{\bf B}$. By Theorem 4.4.1 of \cite{Haas} flag
varieties are quasi $\sf D$--affine, that is every $\mathcal
D$--module on a flag variety is $\mathcal D$-generated by its global
sections. This implies that the $\sf D$-affinity
of ${\bf G}/{\bf B}$ will follow if the sheaf ${\mathcal D}_{{\bf
G}/{\bf B}}$ has vanishing higher cohomology groups. 
The main result of the paper is the following theorem:

\begin{theorem}
$${\rm H}^{i}({\bf G}/{\bf B},{\mathcal D}_{{\bf G}/{\bf B}}) = 0$$
\end{theorem}

for $i>0$.

\begin{proof} Let ${\sf F}^n:{\bf G}/{\bf B}\rightarrow {\bf G}/{\bf
    B}$ be the $n$-th absolute Frobenius morphism. By Theorem 1.2.4 of 
    \cite{Haas} there is an isomorphism of sheaves ${\mathcal D}_{{\bf G}/{\bf
    B}}=\bigcup _{n\geq 1}{\mathcal End}({\sf F}^n_{\ast}\Oo _{{\bf
    G}/{\bf B}})$. Fix $i\geq 0$. Clearly, ${\rm H}^{i}({\bf
    G}/{\bf B},{\mathcal End}({\sf F}^n_{\ast}\Oo _{{\bf G}/{\bf B}}))
  = 0$ for all $n\geq 1$ implies 
${\rm H}^{i}({\bf G}/{\bf B},{\mathcal D}_{{\bf G}/{\bf B}}) = 0$. The
statement will follow from Theorem \ref{th:B_2} below, whose proof
occupies the next two sections.

\end{proof}

\begin{theorem}\label{th:B_2}
$$
{\rm H}^{i}({\bf G}/{\bf B},{\mathcal End}({\sf F}^n_{\ast}\Oo _{{\bf G}/{\bf B}})) = 0
$$
\end{theorem}
for $i>0$ and $n\geq 1$.

\vspace*{0.2cm}

The method used in \cite{Haas} and \cite{AK} was to identify the sheaf  
${\mathcal End}({\sf F}^n_{\ast}\Oo _{{\bf G}/{\bf B}})$ with an
equivariant vector bundle on ${\bf G}/{\bf B}$ associated to the
induced module ${\rm Ind}_{\bf B}^{{\bf G}_n{\bf B}}(2(p^n-1)\rho)$,
where ${\bf G}_n$ is the $n$-th Frobenius kernel, and to study an
appropriate filtration on such a module. Our main tool is a short exact
sequence from \cite{Sam} that relates the Frobenius pushforwards of the
structure sheaves on the total space of a $\Pp ^1$-bundle and on the
base variety.

\section{\bf Proof of Theorem \ref{th:B_2}}

Recall the short exact sequence from \cite{Sam} mentioned above. Assume given a smooth variety $S$ and a locally free sheaf $\Ee$ of
rank 2 on $S$. Let $X = \Pp _S(\Ee)$ be the projective bundle over $S$ 
and $\pi : X\rightarrow S$ the projection. Denote $\Oo _{\pi}(-1)$
the relative invertible sheaf. One has $\pi _{\ast}\Oo _{\pi}(1)=\Ee ^{\ast}$.
\begin{lemma}\label{lem:P^1Lemma}
For any $n\geq 1$ there is a short exact sequence of vector bundles on $X$:
\begin{equation}
0\rightarrow \pi ^{\ast}{\sf F}^n_{\ast}\Oo _{S} \rightarrow
{\sf F}^n_{\ast}\Oo _{X} \rightarrow \pi ^{\ast}({\sf F}^n_{\ast}({\sf
  D}^{p^n-2}\Ee \otimes \mbox{\rm det}\ \Ee)\otimes
\mbox{\rm det} \ \Ee ^{\ast})\otimes \Oo _{\pi}(-1)\rightarrow 0 .
\end{equation}

Here ${\sf D}^{k}\Ee = ({\sf S}^{k}\Ee ^{\ast})^{\ast}$ is the $k$-th
divided power of $\Ee$.
\end{lemma}

For convenience of the reader we recall the proof.

\begin{proof}
Let $\Dd ^b(X)$ be the bounded derived category of coherent sheaves on
  $X$, and denote $[1]$ the shift functor. By \cite{Orlov}, for any object $A\in \Dd ^{b}(X)$ there is a distinguished triangle:
\begin{equation}\label{eq:exacttrianP1bun}
\dots \rightarrow \pi ^{\ast}{\rm R}^{\bullet}\pi _{\ast}A \rightarrow A \rightarrow
\pi ^{\ast}({\tilde A})\otimes \Oo _{\pi}(-1)\rightarrow \pi
^{\ast}{\rm R}^{\bullet}\pi _{\ast}A [1] \rightarrow \dots  .
\end{equation}

\smallskip

The object $\tilde A$ can be found by tensoring the triangle
(\ref{eq:exacttrianP1bun}) with $\Oo _{\pi}(-1)$ and applying the
functor ${\rm R}^{\bullet}\pi _{\ast}$ to the obtained triangle. Given that
${\rm R}^{\bullet}\pi _{\ast}\Oo _{\pi}(-1) = 0$, we get an isomorphism:
\begin{equation}\label{eq:isomforP1bun}
{\rm R}^{\bullet}\pi _{\ast}(A\otimes \Oo _{\pi}(-1)) \simeq
{\tilde A}\otimes {\rm R}^{\bullet}\pi _{\ast}\Oo _{\pi}(-2).
\end{equation}

One has ${\rm R}^{\bullet}\pi _{\ast}\Oo _{\pi}(-2) =
\mbox{det} \ \Ee[-1]$. Tensoring both sides of the isomorphism (\ref{eq:isomforP1bun}) with
$\mbox{det} \ \Ee ^{\ast}$, we get:
\begin{equation}\label{eq:tildeA}
{\tilde A} = {\rm R}^{\bullet}\pi _{\ast}(A\otimes \Oo _{\pi}(-1))\otimes \mbox{det} \
{\mathcal E}^{\ast}[1].
\end{equation}

Let now $A$ be the vector bundle ${\sf F}^n_{\ast}\Oo _X$. The triangle
(\ref{eq:exacttrianP1bun}) becomes in this case:
\begin{equation}\label{eq:trianforFrobdirimage}
\dots \rightarrow \pi ^{\ast}{\rm R}^{\bullet}\pi _{\ast}{\sf F}^n_{\ast}\Oo _{X} \rightarrow
{\sf F}^n_{\ast}\Oo _{X} \rightarrow
\pi ^{\ast}({\tilde A})\otimes \Oo _{\pi}(-1)\rightarrow \pi
^{\ast}{\rm R}^{\bullet}\pi _{\ast}{\sf F}^n_{\ast}\Oo _X [1] \rightarrow
\dots \quad .
\end{equation}

where ${\tilde A} = {\rm R}^{\bullet}\pi _{\ast}({\sf F}_{\ast}\Oo _{X}\otimes \Oo _{\pi}(-1))\otimes \mbox{det} \
{\mathcal E}^{\ast}[1]$. Recall that for a coherent sheaf $\Ff$ on
$X$ one has an isomorphism ${\rm R}^{i}\pi _{\ast}{\sf F}^n_{\ast}\Ff = {\sf F}^n_{\ast}{\rm
  R}^{i}\pi _{\ast}\Ff$, the Frobenius morphism being finite and
commuting with arbitrary morphisms. Therefore,
\begin{equation}
{\rm R}^{\bullet}\pi _{\ast}{\sf F}^n_{\ast}\Oo _{X} =
{\sf F}^n_{\ast}{\rm R}^{\bullet}\pi _{\ast}\Oo _{X} = {\sf F}^n_{\ast}\Oo
_{S}.
\end{equation}

On the other hand, by the projection formula one has ${\rm R}^{\bullet}\pi _{\ast}({\sf
  F}^n_{\ast}\Oo _{X}\otimes \Oo _{\pi}(-1))$ = ${\rm R}^{\bullet}\pi
_{\ast}({\sf F}^n_{\ast}\Oo _{\pi}(-p^n))$ = ${\sf F}^n_{\ast}{\rm R}^{\bullet}\pi _{\ast}\Oo
_{\pi}(-p^n)$. The relative Serre duality for $\pi$ gives:
\begin{equation}
{\rm R}^{\bullet}\pi _{\ast}\Oo _{\pi}(-p^n) = {\sf D}^{p^n-2}\Ee\otimes \mbox{det} \ \Ee[-1].
\end{equation}

Let ${\tilde \Ee}$ be the vector bundle ${\sf D}^{p^n-2}\Ee\otimes \mbox{det} \ \Ee$.
Putting these isomorphisms together we see that the triangle
(\ref{eq:trianforFrobdirimage}) can be rewritten as follows:
\begin{equation}
\dots \rightarrow \pi ^{\ast}{\sf F}^n_{\ast}\Oo _{S} \rightarrow
{\sf F}^n_{\ast}\Oo _{X} \rightarrow \pi ^{\ast}({\sf F}^n_{\ast}{\tilde \Ee}\otimes
\mbox{det} \ \Ee ^{\ast})\otimes \Oo
_{\pi}(-1)\stackrel{[1]}\rightarrow \dots  .
\end{equation}

Therefore, the above distinguished triangle is in fact a
short exact sequence of vector bundles on $X$:
\begin{equation}\label{eq:frobshorexseqP1bun}
0\rightarrow \pi ^{\ast}{\sf F}^n_{\ast}\Oo _{S} \rightarrow
{\sf F}^n_{\ast}\Oo _{X} \rightarrow \pi ^{\ast}({\sf F}^n_{\ast}{\tilde \Ee}\otimes
\mbox{det} \ \Ee ^{\ast})\otimes \Oo _{\pi}(-1)\rightarrow 0.
\end{equation}

\end{proof}

We will use the projection $q:{\bf G}/{\bf B}\rightarrow {\sf Q}_3$ to compute the bundle ${\sf
  F}^n_{\ast}\Oo _{{\bf G}/{\bf B}}$. Applying Lemma
  \ref{lem:P^1Lemma} in this case, we get a short exact sequence:
\begin{equation}
0\rightarrow q^{\ast}{\sf F}^n_{\ast}\Oo _{{\sf Q}_3}\rightarrow {\sf
  F}^n_{\ast}\Oo _{{\bf G}/{\bf B}}\rightarrow q^{\ast}({\sf F}^n_{\ast}({\sf D}^{p^n-2}\Uu _2(-1))\otimes \Oo _{{\sf
  Q}_3}(1))\otimes \Oo _{q}(-1)\rightarrow 0.
\end{equation}

Here $\Oo _q(-1)=\pi ^{\ast}\Oo _{\Pp ^3}(-1)$ is the relative line bundle with respect
  to the projection $q$. Apply the functor $\Hom (-,{{\sf
  F}^n_{\ast}}\Oo _{{\bf G}/{\bf B}})$ to this sequence. Consider
  first the groups $\Ext ^{i}(q^{\ast}{\sf F}^n_{\ast}\Oo _{{\sf Q}_3},{{\sf
  F}^n_{\ast}}\Oo _{{\bf G}/{\bf B}})$. By adjunction we get an
  isomorphism:
\begin{equation}
\Ext ^{i}(q^{\ast}{\sf F}^n_{\ast}\Oo _{{\sf Q}_3},{{\sf
  F}^n_{\ast}}\Oo _{{\bf G}/{\bf B}}) = \Ext ^{i}({\sf F}^n_{\ast}\Oo _{{\sf Q}_3},{\sf F}^n_{\ast}\Oo _{{\sf
    Q}_3}).
\end{equation}

Indeed, ${\rm R}^{\bullet}q _{\ast}{\sf F}^n_{\ast}\Oo _{{\bf G}/{\bf B}} = {\sf F}^n_{\ast}{\rm
  R}^{\bullet}q_{\ast}\Oo _{{\bf G}/{\bf B}} = {\sf F}^n_{\ast}\Oo _{{\sf Q}_3}$.

\begin{lemma}\label{lem:D-affQ_3}
$\Ext ^{i}({\sf F}^n_{\ast}\Oo _{{\sf Q}_3},{\sf F}^n_{\ast}\Oo _{{\sf
    Q}_3})=0$ for $i>0$ and $n\geq 1$.
\end{lemma}

\begin{proof}
For $n=1$ this follows from \cite{SamCRAS}. For quadrics of arbitrary
dimension an explicit decomposition of the Frobenius pushforward of a
line bundle was found in \cite{Lan}; in particular, this implies Lemma \ref{lem:D-affQ_3}.
However, it is worth giving an independent proof that is based on the
argument from \cite{SamCRAS}; the proof of Theorem \ref{th:B_2} is just an extension of it. Recall (Lemma 2.3, {\it
  loc.cit.}) that there is an isomorphism of cohomology groups:
\begin{equation}
\Ext ^{i}({\sf F}^n_{\ast}\Oo _{{\sf Q}_3},{\sf F}^n_{\ast}\Oo _{{\sf
    Q}_3}) = {\rm H}^{i}({\sf Q}_3\times {\sf Q}_3,({{\sf F}^n\times \sf F}^n)^{\ast}(i_{\ast}\Oo
_{\Delta})\otimes (\Oo  _{{\sf Q}_3}\boxtimes \omega _{{\sf Q}_3}^{1-p^n}))
\end{equation}

There is a resolution of the sheaf $i_{\ast}\Oo _{\Delta}$ (Lemma 3.1, \cite{SamCRAS}):
\begin{equation}\label{eq:resofdiag3dimquad}
0\rightarrow \Uu _2\boxtimes \Uu _2(-2)\rightarrow \Psi _2\boxtimes
\Oo _{{\sf Q}_3}(-2)\rightarrow \Psi
_1\boxtimes \Oo _{{\sf Q}_3}(-1)\rightarrow \Oo _{{\sf Q}_3}\boxtimes \Oo
_{{\sf Q}_3}\rightarrow i_{\ast}\Oo _{\Delta}\rightarrow 0,
\end{equation}

This is a particular case of Kapranov's resolution of the diagonal for
quadrics. Put $\Psi _0 = \Oo _{{\sf Q}_3}$ and $\Psi _3 = \Uu
_2$. When $k$ is of characteristic zero, the bundles $\Psi _i$ for $i=1,2$ can explicitly be described as follows: the bundle $\Psi _1$ is
  isomorphic to the restriction of $\Omega ^{1}(1)$ on $\Pp ^4$ to
  ${\sf Q}_3$, the quadric ${\sf Q}_3$ being naturally embedded into
  $\Pp ^4=\Pp ({\sf W})$, and the bundle $\Psi _2$ fits into the short exact sequence
\begin{equation}\label{eq:Psi_2}
0\rightarrow \Omega ^2_{\Pp ^4}(2)\otimes \Oo _{{\sf Q}_3}\rightarrow
\Psi _2\rightarrow \Oo _{{\sf Q}_3}\rightarrow 0.
\end{equation}

Let us check that the same descriptions of $\Psi _1$ and $\Psi _2$ are
valid when the characteristic of $k$ is an odd prime. This amounts to
computation of cohomology groups. Indeed, for any coherent sheaf $\Ee$ on ${\sf Q}_3$ there is a standard
spectral sequence converging to $\Ee$, and whose ${\rm
  E}_1$-term is equal to ${\rm H}^{i}({\sf Q}_3,\Ee \otimes \Oo _{{\sf
    Q}_3}(j))\otimes \Psi _{-j}$ for $j=-2,\dots
0$, and ${\rm H}^{i}({\sf Q}_3,\Ee \otimes \Uu _2(-2))\otimes \Psi
_{-3}$ for $j=-3$. Taking $\Ee$ to be $\Omega _{\Pp ^4}^1(1)\otimes \Oo _{{\sf Q}_3}$
or $\Omega _{\Pp ^4}^2(2)\otimes \Oo _{{\sf Q}_3}$ and computing the
terms of spectral sequence we see that the cohomology groups in
question are the same as in characteristic zero, thus arriving at the above resolutions for
$\Psi _1$ and $\Psi _2$.

Arguing as in the proof of Theorem 3.2 of \cite{SamCRAS}, we conclude
that Lemma \ref{lem:D-affQ_3} follows from the following statement:
\begin{proposition}\label{prop:U_2prop}
 ${\rm H}^{i}({\sf
    Q}_3,{{\sf F}^n}^{\ast}\Uu _2) = 0$ for $i\neq 2$ and $n\geq 1$.
\end{proposition}

\end{proof}

\begin{proof}
Denote $\Oo _{\pi}(-1)$ the relative line bundle with respect to the
projection  $\pi: {\bf G}/{\bf B}={\Pp}({\sf N})\rightarrow {\bf
  G}/{{\bf P}_{\alpha}}=\Pp ^3$. Consider the short exact sequence
\begin{equation}\label{eq:fundseqonX}
0\rightarrow \pi ^{\ast}\Oo _{\Pp ^3}(-1)\rightarrow q^{\ast}\Uu _2\rightarrow \Oo
_{\pi}(-1)\rightarrow 0.
\end{equation}

Applying the functor ${{\sf F}^n}^{\ast}$ to it we get:

\begin{equation}\label{eq:1stbasicseq}
0\rightarrow \pi ^{\ast}\Oo _{\Pp ^3}(-p^n)\rightarrow q^{\ast}{{\sf
    F}^n}^{\ast}\Uu _2\rightarrow \Oo _{\pi}(-p^n)\rightarrow 0.
\end{equation}

First, one has ${\rm H}^{i}({\bf G}/{\bf B},\pi ^{\ast}\Oo _{\Pp
  ^3}(-p^n))$ = ${\rm
  H}^{i}(\Pp ^3,\Oo _{\Pp ^3}(-p^n))$ = $0$ for $i\neq 3$. Let us
show that ${\rm H}^{i}({\bf G}/{\bf B},\Oo _{\pi}(-p^n)) = 0$ for $i<2$. Indeed, ${\rm H}^{0}({\bf G}/{\bf
  B},\Oo _{\pi}(-p^n))=0$, the line bundle $\Oo _{\pi}(-k)$ being not effective for any $k$.
 Let us consider ${\rm H}^1$. One has ${\rm R}^{\bullet}\pi _{\ast}\Oo _{\pi}(-p^n)={\sf D}^{p^n-2}{\sf
  N}[-1]$. Thus, ${\rm H}^{1}({\bf G}/{\bf B},\Oo _{\pi}(-p^n))={\rm
  H}^{0}(\Pp ^3,{\sf D}^{p^n-2}{\sf N})$. For any $k>0$ there is a
short exact sequence on ${\bf G}/{\bf B}$:
\begin{equation}\label{eq:seqforD^kN}
0\rightarrow \Oo _{\pi}(-k)\rightarrow \pi ^{\ast}{\sf D}^{k}{\sf
  N}\rightarrow \pi ^{\ast}{\sf D}^{k-1}{\sf N}\otimes \Oo
  _{\pi}(1)\rightarrow 0.
\end{equation}

It is obtained from the relative Euler sequence 
\begin{equation}
0\rightarrow \Oo _{\pi}(-1)\rightarrow \pi ^{\ast}{\sf N}\rightarrow
\Oo _{\pi}(1)\rightarrow 0
\end{equation}

by taking first its $k$-th symmetric power and then passing to the
dual (since the bundle ${\sf N}$ is symplectic, it is self-dual, that
is ${\sf N}={\sf N}^{\ast}$). We saw above that the line bundle $\Oo
_{\pi}(-k)$ did not have global sections. Using the sequence (\ref{eq:seqforD^kN}) for
$k=p^n-2,p^n-3,\dots ,1$ and descending induction, we see that ${\rm H}^{0}({\bf G}/{\bf B},\pi ^{\ast}{\sf D}^{k}{\sf
  N})={\rm H}^{0}(\Pp ^3,{\sf D}^{k}{\sf N})=0$. This implies ${\rm
  H}^{i}({\sf Q}_3,{{\sf F}^n}^{\ast}\Uu _2)=0$ for $i<2$. By Serre
duality ${\rm H}^3({\sf Q}_3,{{\sf F}^n}^{\ast}\Uu _2)={\rm H}^0({\sf
  Q}_3,{{\sf F}^n}^{\ast}\Uu _2^{\ast}\otimes \omega _{{\sf
    Q}_3})^{\ast}$. Recall that $\omega _{{\sf Q}_3} = \Oo _{{\sf
    Q}_3}(-3)$. Dualizing the sequence (\ref{eq:1stbasicseq}) and
tensoring it with $\omega _{{\sf Q}_3}$, we see that the bundle ${{\sf F}^n}^{\ast}\Uu _2^{\ast}\otimes \omega _{{\sf
    Q}_3}$ is an extension of two line bundles, both of which are
non-effective. Thus, ${\rm H}^3({\sf Q}_3,{{\sf F}^n}^{\ast}\Uu
_2)=0$. Finally, one gets a short exact sequence:
\begin{equation}
0\rightarrow {\rm H}^{2}({\sf Q}_3,{{\sf F}^n}^{\ast}\Uu _2)\rightarrow
{\rm H}^{2}({\bf G}/{\bf B},\Oo _{\pi}(-p^n))\rightarrow {\rm H}^{3}({\bf G}/{\bf B}, \pi
^{\ast}\Oo (-p^n))\rightarrow 0,
\end{equation}

and the statement follows.

\begin{remark}
The (non)-vanishing of the first cohomology group of a line bundle on
arbitrary flag variety was completely determined by H.H.Andersen
(cf. \cite{An}, 2.3). The line bundle $\Ll _{\chi}=\Oo _{\pi}(-p^n)$ corresponds to the
weight $\chi =p^n\omega _{\alpha}-p^n\omega _{\beta}$. Using
Andersen's criterion one immediately checks the vanishing of ${\rm
  H}^1({\bf G}/{\bf B},\Ll _{\chi})$.
\end{remark}

\end{proof}

Next step is the following vanishing:

\begin{lemma}\label{lem:2ndstep}
\begin{equation}
\Ext ^{i}( q^{\ast}({\sf F}^n_{\ast}({\sf
  D}^{p^n-2}\Uu _2(-1))\otimes \Oo _{{\sf Q}_3}(1))\otimes \Oo
  _{q}(-1),{\sf F}^n_{\ast}\Oo _{{\bf G}/{\bf B}}) = 0
\end{equation}
for $i>0$ and $n\geq 1$.

\end{lemma}

Clearly, Lemma \ref{lem:D-affQ_3} and Lemma \ref{lem:2ndstep} will imply
Theorem \ref{th:B_2}.

\begin{proof}
By the projection formula, one has:
\begin{eqnarray}
& \Ext ^{i}( q^{\ast}({\sf F}^n_{\ast}({\sf
  D}^{p^n-2}\Uu _2(-1))\otimes \Oo _{{\sf Q}_3}(1))\otimes \Oo
  _{q}(-1),{\sf F}^n_{\ast}\Oo _{{\bf G}/{\bf B}}) = \\
& = \Ext ^{i}({\sf F}^n_{\ast}({\sf D}^{p^n-2}\Uu _2(-1))\otimes \Oo _{{\sf Q}_3}(1),{\sf F}^n_{\ast}{\sf
  S}^{p^n}\Uu _2^{\ast}). &\nonumber
\end{eqnarray}

Indeed, 
\begin{eqnarray}
& {\rm R}^{\bullet}q _{\ast}({\sf F}^n_{\ast}\Oo _{{\bf G}/{\bf
  B}}\otimes \Oo _q(1)) = {\rm R}^{\bullet}q _{\ast}({\sf F}^n_{\ast}\Oo _{{\bf G}/{\bf
  B}}\otimes \pi ^{\ast}\Oo _{\Pp ^3}(1)) = {\rm R}^{\bullet}q
  _{\ast}{\sf F}^n_{\ast}\pi ^{\ast}\Oo _{\Pp ^3}(p^n) = \\
& = {\sf F}^n_{\ast}{\rm R}^{\bullet}q _{\ast}\pi ^{\ast}\Oo _{\Pp ^3}(p^n)
 = {\sf F}^n_{\ast}{\sf S}^{p^n}\Uu _2^{\ast}. &\nonumber
\end{eqnarray}
Recall that a right adjoint functor to ${\sf F}^n_{\ast}$ on a smooth
variety $X$ over $k$ is given by the formula:
\begin{equation}\label{eq:rightadjointequat}
{{\sf F}^n}^{!}(?) = {{\sf F}^n}^{\ast}(?)\otimes \omega _{X}^{1-p^{n}}.
\end{equation}

Therefore,
\begin{eqnarray}
& \Ext ^{i}({\sf F}^n_{\ast}({\sf
  D}^{p^n-2}\Uu _2(-1))\otimes \Oo _{{\sf Q}_3}(1),{\sf F}_{\ast}{\sf
  S}^{p^n}\Uu _2^{\ast}) = \\ 
& = \Ext ^{i}({\sf
  D}^{p^n-2}\Uu _2(-1),{{\sf F}^n}^{\ast}{\sf F}^n_{\ast}{\sf S}^{p^n}\Uu _2
  ^{\ast}\otimes \Oo _{{\sf Q}_3}(-p^n)\otimes \omega _{{\sf
  Q}_3}^{1-p^n}). &\nonumber
\end{eqnarray}

We have $\Oo _{{\sf Q}_3}(-p^n)\otimes \omega _{{\sf Q}_3}^{1-p^n} = \Oo
 _{{\sf Q}_3}(2p^n-3)$. Finally, 
\begin{equation}
\Ext ^{i}({\sf
  D}^{p^n-2}\Uu _2(-1),{{\sf F}^n}^{\ast}{\sf F}^n_{\ast}{\sf S}^{p^n}\Uu _2
  ^{\ast}\otimes \Oo _{{\sf Q}_3}(2p^n-3)) = {\rm H}^{i}({\sf Q}_3,,{{\sf
  F}^n}^{\ast}{\sf F}^n_{\ast}{\sf S}^{p^n}\Uu _2
  ^{\ast}\otimes {\sf S}^{p^n-2}\Uu _2^{\ast}(2p^n-2)),
\end{equation}

and there is an isomorphism of cohomology groups (Corollary 2.1, \cite{Sam}):
\begin{eqnarray}\label{eq:ExtgroupsforB_2}
& {\rm H}^{i}({\sf Q}_3,{{\sf F}^n}^{\ast}{\sf F}^n_{\ast}{\sf S}^{p^n}\Uu _2
  ^{\ast}\otimes {\sf S}^{p^n-2}\Uu _2^{\ast}(2p^n-2)) = \\ 
& = {\rm H}^{i}({\sf
  Q}_3\times {\sf Q}_3,({\sf F}^n\times {\sf F}^n)^{\ast}(i_{\ast}\Oo
  _{\Delta})\otimes ({\sf S}^{p^n}\Uu _2
  ^{\ast}\boxtimes {\sf S}^{p^n-2}\Uu _2^{\ast}(2p^n-2)). &\nonumber
\end{eqnarray}

Apply ${{\sf F}^n}^{\ast}\times {{\sf F}^n}^{\ast}$ to
the resolution (\ref{eq:resofdiag3dimquad}). Denote ${\rm
  C}^{\bullet}$ the complex, whose terms are ${\rm C}^j =
{{\sf F}^n}^{\ast}\Psi _{-j}\boxtimes {{\sf F}^n}^{\ast}\Oo _{{\sf Q}_3}(j)$ for
$j=-2,-1,0$ and ${\rm C}^{-3}={{\sf F}^n}^{\ast}\Uu _2\boxtimes {{\sf
    F}^n}^{\ast}\Uu _2(-2)$. Tensor ${\rm C}^{\bullet}$ with the
bundle ${\sf S}^{p^n}\Uu _2^{\ast}\boxtimes {\sf S}^{p^n-2}\Uu _2
^{\ast}(2p^n-2)$. Then the complex ${\rm C}^{\bullet}\otimes ({\sf S}^{p^n}\Uu _2^{\ast}\boxtimes {\sf S}^{p^n-2}\Uu _2
^{\ast}(2p^n-2))$ computes the cohomology group in the right hand side of
(\ref{eq:ExtgroupsforB_2}). 

\begin{lemma}\label{eq:mainauxlemma}
${\rm H}^{i}({\sf Q}_3\times {\sf Q}_3,{\rm C}^{j}\otimes ({\sf S}^{p^n}\Uu _2^{\ast}\boxtimes {\sf S}^{p^n-2}\Uu _2
^{\ast}(2p^n-2))=0$ for $i>-j$ and $n\geq 1$.
\end{lemma}

Clearly, this implies ${\rm H}^{i}({\sf
  Q}_3\times {\sf Q}_3,({\sf F}^n\times {\sf F}^n)^{\ast}(i_{\ast}\Oo
  _{\Delta})\otimes ({\sf S}^{p^n}\Uu _2
  ^{\ast}\boxtimes {\sf S}^{p^n-2}\Uu _2^{\ast}(2p^n-2))=0$ for $i>0$
  and $n\geq 1$, and hence
  Lemma \ref{lem:2ndstep}. The proof of Lemma \ref{eq:mainauxlemma} is
  broken up into a series of propositions below.
\end{proof}

\begin{proposition}
${\rm H}^{i}({\sf Q}_3\times {\sf Q}_3,{\rm C}^{j}\otimes ({\sf S}^{p^n}\Uu _2^{\ast}\boxtimes {\sf S}^{p^n-2}\Uu _2
^{\ast}(2p^n-2))=0$ for $i>-j$, where $j=-1,0$ and $n\geq 1$.
\end{proposition}

\begin{proof}

Indeed, ${\sf S}^{k}\Uu _2^{\ast}={\rm
  R}^{\bullet}q_{\ast}\pi ^{\ast}\Oo _{\Pp ^3}(k)$ for $k\geq 0$, and ${\sf S}^{p^n-2}\Uu _2
^{\ast}(2p^n-2)={\rm R}^{\bullet}q_{\ast}\pi ^{\ast}\Oo _{\Pp
  ^3}(p^n-2)\otimes q^{\ast}\Oo _{{\sf Q}_3}(2p^n-2)$. Both line bundles
  $\pi ^{\ast}\Oo _{\Pp ^3}(p^n)$ and $\pi ^{\ast}\Oo _{\Pp
  ^3}(p^n-2)\otimes q^{\ast}\Oo _{{\sf Q}_3}(2p^n-2)$ are effective, so
  using the projection formula, the Kempf vanishing and the K\"unneth formula,
we see immediately that
\begin{equation}
{\rm H}^{i}({\sf Q}_3\times {\sf Q}_3,{\sf S}^{p^n}\Uu _2^{\ast}\boxtimes {\sf S}^{p^n-2}\Uu _2
^{\ast}(2p^n-2)) = 0
\end{equation}

for $i>0$. Further, the bundle $\Psi _1=\Omega ^1_{\Pp ^4}(1)\otimes
\Oo _{{\sf Q}_3}$ has a resolution:
\begin{equation}
0\rightarrow \Omega ^1_{\Pp ^4}(1)\otimes \Oo _{{\sf Q}_3}\rightarrow
{\sf W}^{\ast}\otimes \Oo _{{\sf Q}_3}\rightarrow \Oo _{{\sf
    Q}_3}(1)\rightarrow 0,
\end{equation}

Tensoring this sequence with ${\sf S}^{p^n-2}\Uu _2^{\ast}$ and using
once again the Kempf vanishing and the K\"unneth formula we get:

\begin{equation}
{\rm H}^{i}({\sf Q}_3\times {\sf Q}_3,({{\sf F}^n}^{\ast}\Psi _1\otimes
{\sf S}^{p^n}\Uu _2^{\ast})\boxtimes {\sf S}^{p^n-2}\Uu _2^{\ast}(p^n-2)) = 0.
\end{equation}

for $i>1$. 

\end{proof}

\begin{proposition}
${\rm H}^{i}({\sf Q}_3\times {\sf Q}_3,({{\sf F}^n}^{\ast}\Psi _2\otimes {\sf S}^{p^n}\Uu _2
^{\ast})\boxtimes {\sf S}^{p^n-2}\Uu _2^{\ast}(-2))=0$ for $i>2$ and
$n\geq 1$.
\end{proposition}

\begin{proof}
Propositions \ref{prop:Prop1forB_2} and \ref{prop:Prop2forB_2} below ensure
that for $n\geq 1$ one has ${\rm H}^{i}({{\sf Q}_3},{\sf S}^{p^n-2}\Uu ^{\ast}(-2)) = 0$ for
$i\neq 1$ and ${\rm H}^{i}({{\sf Q}_3},{{\sf F}^n}^{\ast}\Psi _2\otimes {\sf S}^{p^n}\Uu
_2^{\ast}) = 0$ for $i>1$. The K\"unneth formula finishes the proof.
\end{proof}

\begin{proposition}
${\rm H}^{i}({\sf Q}_3\times {\sf Q}_3,({{\sf
  F}^n}^{\ast}\Uu _2\otimes {\sf S}^{p^n}\Uu _2^{\ast})\boxtimes ({{\sf
  F}^n}^{\ast}\Uu _2\otimes {\sf S}^{p^n-2}\Uu _2^{\ast}(-2))=0$ for
  $i>3$ and $n\geq 1$.
\end{proposition}

\begin{proof}
Similarly to the above lemma, Propositions \ref{prop:Prop3forB_2} and
\ref{prop:Prop4forB_2} below imply that for $n\geq 1$ one has ${\rm H}^{i}({{\sf Q}_3},{{\sf
  F}^n}^{\ast}\Uu _2\otimes {\sf S}^{p^n}\Uu _2^{\ast})=0$ for $i>1$
and ${\rm H}^{i}({{\sf Q}_3},{{\sf F}^n}^{\ast}\Uu _2\otimes {\sf
  S}^{p^n-2}\Uu _2^{\ast}(-2))=0$ for $i>2$. We are done by K\"unneth.
\end{proof}

\section{\bf End of the proof}

\begin{proposition}\label{prop:Prop1forB_2}
${\rm H}^{i}({{\sf Q}_3},{\sf S}^{p^n-2}\Uu _2^{\ast}(-2)) = 0$ for
$i\neq 1$ and $n\geq 1$.
\end{proposition}

\begin{proof}
One has  ${\rm H}^{i}({{\sf Q}_3},{\sf S}^{p^n-2}\Uu _2^{\ast}(-2)) =
{\rm H}^{i}({\bf G}/{\bf B},\pi ^{\ast}\Oo _{\Pp ^3}(p^n-2)\otimes q^{\ast}\Oo
_{{\sf Q}_3}(-2))$. Recall that $\omega _{{\bf G}/{\bf B}}=\pi
  ^{\ast}\Oo _{\Pp ^3}(-2)\otimes q^{\ast}\Oo _{{\sf Q}_3}(-2)$. By Serre duality one has
\begin{equation}
{\rm H}^{i}({\bf G}/{\bf B},\pi ^{\ast}\Oo _{\Pp ^3}(p^n-2)\otimes q^{\ast}\Oo
_{{\sf Q}_3}(-2))={\rm H}^{4-i}({\bf G}/{\bf B},\pi ^{\ast}\Oo _{\Pp ^3}(-p^n))^{\ast},
\end{equation}

and the last group is isomorphic to ${\rm H}^{4-i}(\Pp ^3,\Oo _{\Pp
  ^3}(-p^n))^{\ast}$ that can be non-zero only if $i=1$.

\end{proof}

\begin{proposition}\label{prop:Prop2forB_2}
${\rm H}^{i}({{\sf Q}_3},{{\sf F}^n}^{\ast}\Psi _2\otimes {\sf S}^{p^n}\Uu
_2^{\ast}) = 0$ for $i>1$ and $n\geq 1$.
\end{proposition}

\begin{proof}
  Consider the sequence:
\begin{equation}
0\rightarrow \Omega ^2_{\Pp ^4}(2)\otimes \Oo _{{\sf Q}_3}\rightarrow \wedge ^{2}{\sf
  V}^{\ast}\otimes \Oo _{{\sf Q}_3}\rightarrow \Omega ^1_{\Pp
  ^4}(2)\otimes \Oo _{{\sf Q}_3}\rightarrow 0.
\end{equation}

Tensor it with ${\sf S}^{p^n}\Uu _2^{\ast}$. Since ${\rm H}^{i}({\sf
  Q}_3,{\sf S}^{p^n}\Uu _2^{\ast})={\rm H}^i({\bf G}/{\bf B},\pi
  ^{\ast}\Oo _{\Pp ^3}(p^n))=0$ for $i>0$, we see that the statement
  will follow if we show ${\rm H}^{i}({{\sf Q}_3},{{\sf F}^n}^{\ast}\Omega ^1_{\Pp
  ^4}(2)\otimes \Oo _{{\sf Q}_3}\otimes {\sf S}^{p^n}\Uu _2^{\ast}) =
  0$ for $i>0$. Recall that ${\sf Q}_3\subset \Pp ^4 = \Pp ({\sf W})$.
Consider the adjunction sequence tensored with $\Oo _{{\sf Q}_3}(-1)$:
\begin{equation}
0\rightarrow {\mathcal T}_{{\sf Q}_3}(-1)\rightarrow {\mathcal T}_{\Pp ^4}\otimes
\Oo _{{\sf Q}_3}(-1)\rightarrow \Oo _{{\sf Q}_3}(1)\rightarrow 0.
\end{equation}

Recall that if the characteristic $p$ is odd then the bundle ${\mathcal T}_{{\sf Q}_3}(-1)$ is self-dual, that is
${\mathcal T}_{{\sf Q}_3}(-1)=\Omega _{{\sf Q}_3}^1(1)$ on ${\sf Q}_3$
(see, for instance, \cite{Samvan}, Lemma 4.1). Dualizing the above sequence and
tensoring it then with $\Oo _{{\sf Q}_3}(1)$, we get:
\begin{equation}
0\rightarrow \Oo _{{\sf Q}_3}\rightarrow \Omega ^1_{\Pp
  ^4}(2)\otimes \Oo _{{\sf Q}_3}\rightarrow {\mathcal T}_{{\sf Q}_3}\rightarrow
  0.
\end{equation}
 Consequently, the statement will follow from 
${\rm H}^{i}({{\sf Q}_3},{{\sf F}^n}^{\ast}{\mathcal T}_{{\sf Q}_3}\otimes {\sf S}^{p^n}\Uu _2^{\ast}) =
  0$ for $i>0$. Since $p$ is odd, one has ${\mathcal T}_{{\sf Q}_3} = {\sf S}^{2}\Uu
  _2^{\ast}$. Consider the universal exact sequence:
\begin{equation}\label{eq:oncemoreunivseq}
0\rightarrow \Uu _2\rightarrow {\sf V}\otimes \Oo _{{\sf
    Q}_3}\rightarrow \Uu _2^{\ast}\rightarrow 0.
\end{equation}

Recall that $\mbox{det} \ \Uu _2 = \Oo _{{\sf Q}_3}(-1)$. Taking the symmetric square of this sequence and then applying the
functor ${{\sf F}^n}^{\ast}$, we get:
\begin{equation}
0\rightarrow \Oo _{{\sf Q}_3}(-p^n)\rightarrow {{\sf F}^n}^{\ast}\Uu
_2\otimes {{\sf F}^n}^{\ast}{\sf V}\rightarrow {{\sf F}^n}^{\ast}{\sf S}^2{\sf
  V}\otimes \Oo _{{\sf Q}_3}\rightarrow {{\sf F}^n}^{\ast}{\sf S}^{2}\Uu _2^{\ast}\rightarrow 0.
\end{equation}

Tensor it with ${\sf S}^{p^n}\Uu _2^{\ast}$. Proposition \ref{prop:Prop3forB_2} below states that ${\rm H}^{i}({\sf Q}_3,{\sf S}^{p^n}\Uu _2^{\ast}\otimes {{\sf
    F}^n}^{\ast}\Uu _2) = 0$ for $i>1$ and $n\geq 1$. It is sufficient
    therefore to show that ${\rm H}^{3}({\sf Q}_3,{\sf S}^{p^n}\Uu _2^{\ast}\otimes \Oo
_{{\sf Q}_3}(-p^n)) = 0$, or, equivalently, that ${\rm H}^{3}({\bf G}/{\bf
  B},\Oo _{\pi}(-p^n)) = 0$. Indeed, there is an isomorphism of line bundles:
\begin{equation}\label{eq:relbetlinebun}
q^{\ast}\Oo _{{\sf Q}_3}(p^n) =  \Oo _{\pi}(p^n)\otimes \pi ^{\ast}\Oo
_{\Pp ^3}(p^n).
\end{equation}

Hence, by the projection formula:
\begin{equation}
{\rm H}^{3}({\sf Q}_3,{\sf S}^{p^n}\Uu _2^{\ast}\otimes \Oo
_{{\sf Q}_3}(-p^n))={\rm H}^{3}({\sf Q}_3,q_{\ast}(\pi ^{\ast}\Oo _{\Pp
  ^3}(p^n)\otimes q^{\ast}\Oo _{{\sf Q}_3}(-p^n)) = {\rm H}^{3}({\bf G}/{\bf
  B},\Oo _{\pi}(-p^n)).
\end{equation}

Considering the sequence
(\ref{eq:1stbasicseq}) and applying Proposition \ref{prop:U_2prop} we get the statement.
\end{proof}

\begin{proposition}\label{prop:Prop3forB_2}
${\rm H}^{i}({\sf Q}_3,{\sf S}^{p^n}\Uu _2^{\ast}\otimes {{\sf
    F}^n}^{\ast}\Uu _2) = 0$ for $i>1$ and $n\geq 1$.
\end{proposition}

\begin{proof}
Apply ${{\sf F}^n}^{\ast}$ to the sequence (\ref{eq:oncemoreunivseq}):
\begin{equation}\label{eq:F^*oftautseqonQ_3}
0\rightarrow {{\sf F}^n}^{\ast}\Uu _2\rightarrow {{\sf F}^n}^{\ast}{\sf V}\otimes \Oo _{{\sf Q}_3}
\rightarrow {{\sf F}^n}^{\ast}\Uu _2^{\ast}\rightarrow 0.
\end{equation}

Tensoring the sequence (\ref{eq:F^*oftautseqonQ_3}) with
${\sf S}^{p^n}\Uu _2^{\ast}$, we obtain:
\begin{equation}
0\rightarrow {{\sf F}^n}^{\ast}\Uu _2\otimes {\sf S}^{p^n}\Uu _2
^{\ast}\rightarrow {{\sf F}^n}^{\ast}{\sf V}\otimes {\sf S}^{p^n}\Uu _2
^{\ast}\rightarrow {{\sf F}^n}^{\ast}\Uu _2^{\ast}\otimes {\sf S}^{p^n}\Uu _2
^{\ast}\rightarrow 0
\end{equation}

We saw above that the higher cohomology groups of ${\sf S}^{p^n}\Uu _2^{\ast}$ vanish. Hence, from the long
exact cohomology sequence it is sufficient to show
that ${\rm H}^{i}({\sf Q}_3,{\sf S}^{p^n}\Uu _2^{\ast}\otimes {{\sf
  F}^n}^{\ast}\Uu _2^{\ast}) = 0$ for $i>0$. Take the dual to the sequence (\ref{eq:1stbasicseq}):
\begin{equation}\label{eq:decompforF^*U^*}
0\rightarrow \Oo _{\pi}(p^n)\rightarrow q^{\ast}{{\sf F}^n}^{\ast}\Uu _2
^{\ast}\rightarrow \pi ^{\ast}\Oo _{\Pp ^3}(p^n)\rightarrow 0. 
\end{equation}

Tensor this sequence with $\pi ^{\ast}\Oo _{\Pp ^3}(p^n)$:
\begin{equation}
0\rightarrow \Oo _{\pi}(p^n)\otimes \pi ^{\ast}\Oo _{\Pp ^3}(p^n)\rightarrow \pi
^{\ast}\Oo _{\Pp ^3}(p^n)\otimes q^{\ast}{{\sf F}^n}^{\ast}\Uu _2
^{\ast}\rightarrow \pi ^{\ast}\Oo _{\Pp ^3}(2p^n)\rightarrow 0.
\end{equation}

Using the isomorphism (\ref{eq:relbetlinebun}) we get: 
\begin{equation}
0\rightarrow q^{\ast}\Oo _{{\sf Q}_3}(p^n)\rightarrow \pi
^{\ast}\Oo _{\Pp ^3}(p^n)\otimes q^{\ast}{{\sf F}^n}^{\ast}\Uu _2
^{\ast}\rightarrow \pi ^{\ast}\Oo _{\Pp ^3}(2p^n)\rightarrow 0.
\end{equation}

Applying to this sequence the functor $q_{\ast}$, and using the
projection formula and the isomorphism ${\rm R}^{\bullet}q_{\ast}\pi
^{\ast}\Oo _{\Pp ^3}(k)={\sf S}^{k}\Uu _2^{\ast}$ for $k\geq 0$, we obtain:
\begin{equation}
0\rightarrow \Oo _{{\sf Q}_3}(p^n)\rightarrow {\sf S}^{p^n}\Uu
_2^{\ast}\otimes {{\sf F}^n}^{\ast}\Uu _2
^{\ast}\rightarrow {\sf S}^{2p^n}\Uu _2^{\ast}\rightarrow 0.
\end{equation}

The leftmost and rightmost terms of the above sequence have
vanishing higher cohomology, hence the statement of the lemma.
\end{proof}

\begin{proposition}\label{prop:Prop4forB_2}
${\rm H}^{3}({\sf Q}_3,{\sf S}^{p^n-2}\Uu _2^{\ast}(-2)\otimes {{\sf
  F}^n}^{\ast}\Uu _2) =0$ for $n\geq 1$.
\end{proposition}

\begin{proof}
One has ${\rm H}^{i}({{\sf Q}_3},{\sf S}^{p^n-2}\Uu
_2^{\ast}(-2)\otimes {{\sf F}^n}^{\ast}\Uu _2) =
{\rm H}^{i}({\bf G}/{\bf B},\pi ^{\ast}\Oo _{\Pp ^3}(p^n-2)\otimes q^{\ast}\Oo
_{{\sf Q}_3}(-2)\otimes q^{\ast}{{\sf F}^n}^{\ast}\Uu _2)$. By Serre
duality we get:
\begin{equation}
{\rm H}^{3}({\bf G}/{\bf B},\pi ^{\ast}\Oo _{\Pp ^3}(p^n-2)\otimes q^{\ast}\Oo
_{{\sf Q}_3}(-2)\otimes q^{\ast}{{\sf F}^n}^{\ast}\Uu _2)={\rm
  H}^1({\bf G}/{\bf B},\pi ^{\ast}\Oo _{\Pp ^3}(-p^n)\otimes q^{\ast}{{\sf F}^n}^{\ast}\Uu _2^{\ast})^{\ast}.
\end{equation}

Tensor the sequence (\ref{eq:F^*oftautseqonQ_3}) with $\pi
^{\ast}\Oo _{\Pp ^3}(-p^n)$. One has:
\begin{eqnarray}
& \dots \rightarrow {\rm H}^{1}({\bf G}/{\bf B},{{\sf F}^n}^{\ast}{\sf V}\otimes \pi
^{\ast}\Oo _{\Pp ^3}(-p^n))\rightarrow {\rm
  H}^1({\bf G}/{\bf B},\pi ^{\ast}\Oo _{\Pp ^3}(-p^n)\otimes
q^{\ast}{{\sf F}^n}^{\ast}\Uu _2^{\ast})\rightarrow \\
& \rightarrow {\rm
  H}^{2}({\bf G}/{\bf B},\pi ^{\ast}\Oo _{\Pp ^3}(-p^n)\otimes q^{\ast}{{\sf
    F}^n}^{\ast}\Uu _2)\rightarrow \dots . &\nonumber
\end{eqnarray}

Clearly, ${\rm H}^{1}({\bf G}/{\bf B},{{\sf F}^n}^{\ast}{\sf V}\otimes \pi
^{\ast}\Oo _{\Pp ^3}(-p^n)) \ = \ {\rm H}^1(\Pp ^3,\Oo _{\Pp
  ^3}(-p^n))\otimes {{\sf F}^n}^{\ast}{\sf V} = 0$. Let us show that ${\rm
  H}^{2}({\bf G}/{\bf B},\pi ^{\ast}\Oo _{\Pp ^3}(-p^n)\otimes q^{\ast}{{\sf
    F}^n}^{\ast}\Uu _2)$ = $0$. Indeed, tensoring the sequence
(\ref{eq:1stbasicseq}) with $\pi ^{\ast}\Oo _{\Pp ^3}(-p^n)$ we see that ${\rm
  H}^{i}({\bf G}/{\bf B},\pi ^{\ast}\Oo _{\Pp ^3}(-p^n)\otimes q^{\ast}{{\sf
    F}^n}^{\ast}\Uu _2)=0$ if $i\neq 3$, hence the statement.

\end{proof}

\end{document}